\documentclass[pagesize]{scrartcl}
\usepackage{graphicx}
\usepackage{enumerate}
\usepackage{amssymb, amsmath, amsthm}
\usepackage[round]{natbib}
\usepackage{pgfplots}

\usepackage{hyperref}
\usepackage{tikz}
\usetikzlibrary{arrows,automata}

\usepackage[final]{changes} % Comments are ignored
\newtheorem{proposition}{Proposition}[section] %denna med i latex

\theoremstyle{remark}
\newtheorem{remark}[proposition]{Remark}

\usepackage[T1]{fontenc}
\usepackage{lmodern}

\newcommand{\A}{\ensuremath{{\mathcal A}}}
\newcommand{\N}{\ensuremath{{\mathbb N}}}

\begin{document}
\title{\replaced{The Monotone Case Approach for the Solution of Certain Multidimensional Optimal Stopping Problems}{\newline The Monotone Case Approach --
	A Review and Its Application to Certain Multidimensional Optimal Stopping Problems}}
%\title{An Elementary Method for the Explicit Solution of Multidimensional Optimal Stopping Problems}
%\title{On explicit solutions of optimal stopping problems for strong Markov processes}
\author{S\"oren Christensen\thanks{Christian-Albrechts-Universit\"at, Mathematisches Seminar, Ludewig-Meyn-Str. 4, 24098 Kiel, Germany}
, Albrecht Irle\footnotemark[1]}
\date{\today}
\maketitle

\begin{abstract}
%	 We study explicitly solvable multidimensional optimal stopping problems. Our approach is based on the notion of monotone stopping problems in discrete and continuous time. The method is illustrated with a variety of examples including multidimensional versions of the house-selling and burglar's problem, the Poisson disorder problem, and an optimal investment problem. 
This paper studies explicitly solvable multidimensional optimal stopping problems \added{ of sum- and product-type} in discrete and continuous time using the monotone case approach. It gives a review on monotone case stopping using\deleted{ as a novel feature} the Doob decomposition, resp. Doob-Meyer decomposition in continuous time, also in its multiplicative versions. The approach via these decompositions leads to explicit solutions for a variety of examples, including multidimensional versions of the house-selling and burglar's problem, the Poisson disorder problem, and an optimal investment problem. 
\end{abstract}


\vspace{.2cm}

\section{Introduction}
In multidimensional problems, optimal stopping theory reaches its limits when trying to find explicit solutions for problems with a finite time horizon or an underlying (Markovian) process in dimension $d\geq 2$. In the one-dimensional case with infinite time horizon, the optimal continuation set usually is an interval of the real line, bounded or unbounded, so it remains to determine the boundary of that interval, which boils down to finding equations for one or two, resp., real numbers. A wealth of techniques has been developed to achieve this, see \cite{salminen85}, \cite{Karatzas} for one dimensional diffusions, or \cite{MoSa} and \cite{CST} for jump processes, to name but a few. 

%In the multidimensional case, the optimal continuation set is an open subset of a $d$-dimensional space, so its boundary is usually described by some $(d-1)$-dimensional surface. 
%[Beiblatt 2a] 
The notion of a multidimensional stopping problem is employed in this article in the following sense. It is used for such problems where the general theory of optimal stopping prescribes that the optimal stopping time is given by the first entrance time of a stochastic process into a $d$-dimensional optimal stopping set with (Euclidean) dimension $d \geq 2$. Some typical cases are as follows: The underlying stochastic process is a Markovian process with $d$-dimensional state space; under certain explicit time dependencies of the pay-off, in particular for finite time horizon, the space-time process arises with $d = 2$. Particularly in discrete time problems multidimensional problems arise due to history dependence for the optimal stopping time, e.g. for a Markovian process of degree $k$, or even full history dependence as in the well-known best choice problem of Robbins\added{, see \cite{MR2144897}}. 
There are a few multidimensional problems, see \cite{dubins-et-al}, \cite{margrabe}, \cite{gerber-shiu}, \cite{shepp-shiryaev:1994}, which, by some transformation method, may be transferred to a one-dimensional problem. Let us call a multidimensional problem truly multidimensional (as a manner of speech) if such a transformation seems hardly possible, at least does not seem to be available in the current literature (as we know it).
% We do not have any knowledge of such problems in the literature which feature an explicit solution
Explicit solutions for such problems seem to be rare, but, of course, many techniques have been developed to tackle such problems, either semi-explicitly using nonlinear integral equations, see the monograph \cite{ps} or the more recent article \cite{christensen2016optimal} for an overview, or numerically, see Chapter 8 in \cite{De} and \cite{Gla}. 

The purpose of this note is to provide some examples of seemingly truly multidimensional problems with an explicit solution\replaced{ where the payoffs take the form $\big(\sum_{i=1}^mX^i_n\big)_n$ or $\big(\prod_{i=1}^mX^i_n\big)_n$ for $m$ stochastic processes $(X_n^1)_n,\dots,(X_n^m)_n$ in discrete time as well as in continuous time. Knowing the individual solutions for the $(X_n^i)_n$-problems in general does not seem to lead to the explicit solution for the sum or product problem. Here, we present a class of examples for which this is possible, the key being the notion of monotone stopping problems.}{. The key for this is the notion of monotone stopping problems. }
The class of monotone stopping problems has been used extensively in the solution of optimal stopping problems, in particular in the first decades starting with \cite{MR0132593,MR0157465}. A long list of examples can be found in \cite{CRS} and, more recently, in \cite{Ferguson}. The extension to continuous time problems is not straightforward. This was developed in \cite{ross1971infinitesimal}, \cite{MR533005}, \cite{MR731723}, and \cite{MR1030003}. Although these references are not very recent, it is interesting to note that the solution to certain ``modern'' optimal stopping problems is directly based on the notion of monotone case problems.
Here, one may look e.g. at the odds-algorithm initiated in \cite{MR1797879} and extended by \cite{MR3557090}, or at \cite{doi:10.1080/07474946.2016.1275314} where, to solve the original problem, an auxiliary problem for a two-dimensional process consisting of the underlying Markov process and its running maximum with suitable monotonicity properties is introduced\added{. See also \cite{doi:10.1080/17442508.2018.1541991} for related results}.
%, see, e.g., the odds-algorithm initiated in \cite{MR1797879}, or can be interpreted this way, see \cite{doi:10.1080/07474946.2016.1275314} and also the discussion at the end of Appendix \ref{append:mult_optimal}. 

%The structure of this paper is as follows: We start by reviewing monotone case problems in terms of the Doob(-Meyer) decomposition in Section \ref{sec:monot} and state criteria for optimality of the myopic stopping time in two appendices. We then present implications for multidimensional optimal stopping problems in Section \ref{sec:multi}. 
%[Beiblatt 2b]
This paper aims at presenting the monotone case approach to optimal stopping by a systematic use of the Doob (in continuous time Doob-Meyer) decomposition of the pay-off process, and shows how it may be used to solve some multidimensional stopping problems. The Doob(-Meyer) decomposition is a well-known tool in the treatment of super-/sub-martingales, in particular in optimal stopping it is applied to the Snell envelope. 
In the theory of continuous time finance, the decomposition, applied to the Snell envelope, is used to obtain duality results for option pricing, see \cite{jamshidian2007duality} for an overview and further discussions.

Here we review monotone case problems in terms of the Doob(-Meyer) decomposition with regard to the optimality of the myopic stopping time in Section \ref{sec:monot}. Although this approach 
seems to be natural and is mathematically straightforward, we could not locate it in the literature in this form, so we present it in a survey style\replaced{. Due}{, where due} to the simplicity of this approach,\replaced{ the very short proofs are given. Almost sure finiteness of stopping times is not needed in this approach. The}{, all the proofs are given and we note that the} multiplicative Doob(-Meyer) decomposition is included, and our treatment covers the discrete and continuous time case in a unified way \added{such }that it can be used for the application in multidimensional problems in Sections \ref{sec:multi} and \ref{sec:ex}.
\replaced{In Section \ref{sec:multi} we show that, under certain assumptions, the monotone case property of the individual stopping problems carries over to the sum- and product-type problems providing an explicit solution. For this, we use the Doob(-Meyer) decomposition. 
%	As examples in Section \ref{sec:ex} show it is not true in general that the individual monotone case property carries over 
}{In Section \ref{sec:multi}, we observe that, under certain assumptions, the monotonicity property of the special individual underlying processes carries over to sum- and product-type problems, which makes these  also solvable. }
%To convince the reader that this elementary line of argument is nonetheless useful, w
We \deleted{then }discuss a variety of examples in Section \ref{sec:ex}. We start with multidimensional versions of the classical house-selling and burglar's problem. Here, the original one-dimensional problems are well-known to be solvable using the theory of monotone stopping. 
Also the multidimensional house-selling problem with recall was already solved in \cite{bruss1997multiple}.
The last two examples are multidimensional extensions of continuous-time problems: 
%which are traditionally solved using other arguments
the Poisson disorder problem and the optimal investment problem, which in one dimension are usually solved using other argument. The more involved arguments concerning the Doob-Meyer decomposition are treated there in detail.

\section{Monotone Stopping Problems}\label{sec:monot}
\subsection{Monotone stopping problems in discrete time and the Doob decomposition}\label{subsec:mono_discr}
Let us first stay in the realm of discrete time problems with infinite time horizon. For a sequence $X_1,X_2,\dots$ of integrable random variables,  adapted to a given filtration\ $(\A_n)_{n\in\N}$, we want to find a stopping time $\tau^*$ such that
\begin{align}\label{eq:OSP}
	EX_{\tau^*}=\sup_\tau EX_\tau.
\end{align}
Here $\tau$ runs through all stopping times such that $EX_\tau$ exists. We include a random variable $X_\infty$, so that the stopping times may assume the value $\infty$. A natural choice for our problems below is 
$X_\infty=\liminf_{n\rightarrow\infty}X_n,$
see the discussion in Subsection \ref{subsec:ass}. 

There is a certain class of such problems for which we can easily solve this.
Call the above problem a monotone case problem iff for all $n\in\N$ it holds that
\[E(X_{n+1}|\A_n)\leq X_n\;\; \implies E(X_{n+2}|\A_{n+1})\leq X_{n+1}.\]
Using sets in the notation this may be written as
\[\{E(X_{n+1}|\A_n)\leq X_n\}\subseteq \{E(X_{n+2}|\A_{n+1})\leq X_{n+1}\}\mbox{ for all }n\in\N. \]
A particularly simple sufficient condition for the monotone case is that the differences 
\[Y_n=E(X_{n+1}|\A_{n})-X_n\mbox{ 
	are non-increasing in $n$,}\]
	 hence 
	 \[Y_n\leq 0\; \implies Y_{n+1}\leq 0.\]
This condition turns out to be fulfilled in many examples of interest and allows for the treatment of multidimensional problems of sum type discussed in the following section. 

If we only want to consider stopping times $\geq k$, e.g. if stopping in $\{1,...,k-1\}$ is clearly suboptimal, then we may formulate the monotone case condition only for $n\geq k$. Of course, using $X_n'=X_{n+(k-1)}$, this may be subsumed in the case $k=1$, so we shall look at\ $n\geq 1$ in the following and, similarly in the latter continuous time case, at $t\geq0$.

\replaced{Comparing the current gain with what to expect in the next step leads to the stopping time
\[\tau^*=\inf\{n:X_n\geq E(X_{n+1}|\A_n)\}=\{n:Y_n\leq 0\}.\]
It is called the one-step look ahead rule or, as we will use in this paper, the myopic rule. In general, this does not yield an optimal rule, but in monotone case problems it is the natural candidate for an optimal one. 
 }{For monotone case problems, it is natural to consider the myopic stopping time, defined as ...}
%\[\tau^*=\inf\{n:X_n\geq E(X_{n+1}|\A_n)\}=\{n:Y_n\leq 0\}.\]
%Then, we have the monotone case with 
%\[\tau^*=\inf\{k:Y_k\leq 0\}.\]
%\subsection{Monotone stopping problems and the Doob decomposition}
%\subsubsection{The discrete time case}\label{subsec:mono_discr}
The discussion of the optimality of $\tau^*$ \added{for the monotone case }is a well-known topic, see the references mentioned in the introduction.
We, however, find it enlightening to
% to discuss the basic idea in terms of the Doob decomposition. 
%  We now 
  provide a short review using the Doob decomposition, which leads to a shortcut to optimality results without the usual machinery of optimal stopping theory. This approach also provides a unifying line of argument for both discrete and continuous time. 
%This optimality may easily been seen by using the Doob decomposition for an integrable process $(X_n)_{n\in\N}$.
 For every $n$ let
\begin{align*}
M_n&=\sum_{k=1}^{n-1}(X_{k+1}-E(X_{k+1}|\A_{k})),\,\,M_1=0,\\
A_n&=\sum_{k=1}^{n-1}(E(X_{k+1}|\A_{k})-X_k)=\sum_{k=1}^{n-1}Y_k,\,\,A_1=0,
\end{align*}
so that by telescoping expectation terms we have the Doob decomposition
\[X_n=X_1+M_n+A_n\]
with a zero mean martingale $(M_n)_{n\in\N}$. 
For the myopic stopping time $\tau^*$
\[E(X_{k+1}|\A_{k})-X_k>0\mbox{ for }k=1,\dots,\tau^*-1\]
-- valid for all $k\in\N$ if $\tau^*=\infty$ -- and 
in the monotone case
\[E(X_{k+1}|\A_{k})-X_k\leq 0\mbox{ for }k=\tau^*,\tau^*+1,\dots .\]
Thus we have
\[A_{\tau^*}=\sup_n A_n\]
and, using $\tau^*_L=\min\{\tau^*,L\}$ for $L\in\N$,
\[A_{\tau^*_L}=\sup_{n\leq L} A_n.\]
So, for any stopping time $\tau$, not necessarily finite a.s., 
\[A_\tau\leq A_{\tau^*},\;\; A_{\min\{\tau,L\}}\leq A_{\tau^*_L}.\]
	Basically these simple inequalities \replaced{are the foundation for}{lead to} the optimality properties of the myopic stopping time. We provide sufficient conditions\added{ for this optimality}, called \eqref{eq:V1} and \eqref{eq:V2}, which are suitable for our classes of problems, in Subsection \ref{append:doob_optimal}.

\begin{remark}\label{rem:finite_time}
It is remarkable that the myopic stopping rule immediately provides optimal stopping times for all possible time horizons in the monotone case. The same observation also holds true in the continuous time case discussed below. See also \cite{Ferguson} and \cite{doi:10.1080/07474946.2016.1275319}. This is in strong contrast to most Markovian-type optimal stopping problems, where infinite time problems are often easier to solve as the stopping boundary is not time dependent. 
\end{remark}
%	 this problem is often easier to solve than the corresponding problem with finite time horizon where the boundary becomes time dependent. 

%\begin{remark}\label{rem:finite_time}
% The optimal stopping problem \eqref{eq:OSP} is usually formulated as an  infinite time horizon problem. In Markovian-type optimal stopping problems this problem is often easier to solve than the corresponding problem with finite time horizon where the boundary becomes time dependent. This is, however, not the case for monotone optimal stopping problems. As discussed in ..., it can be seen that if $L<\infty$ is the time horizon of a monotone problem, then
% \[\min\{\tau^*,L\}\]
% is optimal, where $\tau^*$ is the myopic stopping time. 
%\end{remark}

	\subsection{Optimality of The Myopic Rule Based on the Doob Decomposition}\label{append:doob_optimal}
We continue the treatment of Subsection \ref{subsec:mono_discr} and
%The following considerations are based on the setting in Subsection \ref{subsec:mono_discr}. In particular, we 
assume that we are in the monotone case. 
\subsubsection{Optimality for finite time horizon}	\label{subsec:opt_finite_add}
Let\ $L\in\N$ and $\tau\leq L$ a bounded stopping time. Then, using the martingale property, $EM_\tau=0$, valid for bounded stopping times, 
\[EX_\tau=EX_1+EA_\tau\leq EX_1+EA_{\tau^*_L}=EX_{\tau^*_L}.\]
%	We obtain
%	\[EX_{\tau^*_L}=\sup_{\tau\leq L}EX_\tau. \]
\replaced{This implies }{Hence, }optimality of $\tau^*_L$ for the finite time horizon $L$\deleted{ follows}. 

\subsubsection{Optimality for infinite time horizon} 
The extension from the finite to the infinite case uses the approximation $\tau=\lim_{L\rightarrow\infty}\tau_L,$ $\tau_L=\min\{\tau,L\},$ so that $X_\tau=\lim_{L\rightarrow\infty}X_{\tau_L}$ on\ $\{\tau<\infty\}$, but on $\{\tau=\infty\}$ we need a specific definition of $X_\infty.$ Here, we use 
\[X_\infty=\liminf_{n\rightarrow\infty}X_n,\]
 so that $X_\tau=\liminf_{L\rightarrow\infty}X_{\tau_L}$.
%First note that we may not use $EM_{\tau^*}=0$ as $\tau^*$ is of course not a bounded stopping time in general. Under the condition
We introduce the following conditions:
\begin{equation}\label{eq:V1}\tag{V1}
\lim_{L\rightarrow\infty}EX_{\tau^*_L}\leq EX_{\tau^*}
\end{equation}
\begin{equation}\label{eq:V2}\tag{V2}
\sup\{EX_{\tau}:\tau\mbox{ bounded}\}=\sup_{\tau}EX_{\tau}
\end{equation}
\added{Note for \eqref{eq:V1} that $EX_{\tau^*_L}$ is increasing in $L$.}
\begin{proposition}\label{prop:opt_myopic_discr}
	Under \eqref{eq:V1} and \eqref{eq:V2}, $\tau^*$ is optimal, i.e.
%	we have optimality.
	\[EX_{\tau^*}=\sup_{\tau}EX_\tau.\]
\end{proposition}
\begin{proof}
	We have from the previous considerations that
	\[EX_{\tau^*}\geq \lim_{L\rightarrow\infty}\sup_{\tau\leq L}EX_{\tau}=\sup\{EX_{\tau}:\tau\mbox{ bounded}\},\]
	proving the claim by \eqref{eq:V2}.
%	In addition, we formulate the condition...
\end{proof}

\subsubsection{Discussion of Assumptions \eqref{eq:V1} and \eqref{eq:V2}}\label{subsec:ass}

The validity of \eqref{eq:V2} is of course a well-known topic in optimal stopping, independently of the monotone case context. We only remark that, under $E(\inf_n X_n)>-\infty$, Fatou's Lemma shows that for any $\tau$
\[\liminf_{L\rightarrow\infty}EX_{\tau_L}\geq E(\liminf_{L\rightarrow\infty}X_{\tau_L})=EX_\tau. \]
The same holds if we add costs of observation, e.g. $X_n=X_n'-cn$, assuming $E(\inf_n X_n')>-\infty$.

\eqref{eq:V1} follows from the condition $E(\sup_n X_n)<\infty,$ which is the standard assumption in optimal stopping theory\added{, see e.g. Theorem 4.5, 4.5' in \cite{CRS}}. This needs a short argument.
\begin{proposition}
	$E(\sup_n X_n)<\infty \added{\implies}$\deleted{implies} \eqref{eq:V1}.
\end{proposition}
\begin{proof}
 Using Fatou's Lemma again, we have
\[E\limsup_{L\rightarrow\infty}X_{\tau^*_L}\geq \limsup_{L\rightarrow\infty}EX_{\tau^*_L}.\]
Due to our definition of $X_\infty$ we have to show that $\limsup_{L\rightarrow\infty}X_{\tau^*_L}=\liminf_{L\rightarrow\infty}X_{\tau^*_L}$ on $\{\tau^*=\infty\}$. On this set, $(A_n)_{n\in\N}$ is increasing, hence $\lim_{n\rightarrow\infty}A_n$ exists. Furthermore, $(M_{\tau^*_n})_{n\in\N}$ is a martingale fulfilling the boundedness condition
\[M_{\tau^*_n}=X_{\tau^*_n}-X_1-A_{\tau^*_n}\leq \sup_n X_n+|X_1|,\]
since $A_{\tau^*_n}\geq 0.$ We may thus invoke the martingale convergence theorem and obtain the convergence of $M_{\tau^*_n}$ to some a.s. finite random variable. Since $\tau^*_n=n$ on $\{\tau^*=\infty\}$, this shows the convergence of\ $(M_n)_n$, hence of $(X_n)_n$, on this set. 
\end{proof}

\subsection{Monotone stopping problems in continuous time and the Doob-Meyer decomposition}\label{subsubsec:doob_cont}
%\subsubsection{The continuous time case}
To find the extension of the discrete time case to continuous time processes $(X_t)_{t\in[0,\infty)}$ we may use the Doob-Meyer decomposition. Under regularity assumptions, not discussed here, we have
\[X_t=X_0+M_t+A_t,\]
where $(M_t)_{t\in[0,\infty)}$ is a zero mean martingale and $(A_t)_{t\in[0,\infty)}$ is of locally bounded variation. Now assume that we may write 
\[A_t=\int_0^tY_sdV_s\]
where $(V_t)_{t\in[0,\infty)}$ is increasing. Then the myopic stopping time -- here often called infinitesimal look ahead rule -- becomes
\[\tau^*=\inf\{t:Y_t\leq 0\}.\]
In this situation, we say that the monotone case holds if 
\[Y_t\leq 0\mbox{ for }t>\tau^*.\]
If $(Y_t)_{t\in[0,\infty)}$ is non-increasing in $t$, then again the monotone case property is immediate. The discussion of optimality is essentially the same as in the discrete time case, so is omitted. As no confusion can occur, we keep the notations \eqref{eq:V1} and \eqref{eq:V2} for the continuous-time versions of the optimality conditions.

%\begin{remark}
%	In many approaches to optimal stopping theory, the smooth fit condition plays a major role. We will not use it in this paper, but let us mention a connection to our approach: Assume that we are in a Markovian situation, i.e. $X_t=g(Z_t)$ for a -- say one-dimensional -- Markov process $Z$, $Y_t=h(Z_t)$ and $V_t=t$. For $Y$ to be monotone, let us say that $Z$ is non-increasing and $h$ is increasing and $z^*$ is root of $h$. Then  $\tau^*:=\inf\{t:Z_t\leq z\}$ and -- under natural assumption --
%	\begin{align*}
%	\frac{E(g(Z_{\tau^*})|Z_0=z^*+\epsilon)-g(z^*)}{\epsilon}&=\frac{E(g(Z_0)+\int_{0}^{\tau^*}h(Z_s)ds|Z_0=z^*+\epsilon)-g(z^*)}{\epsilon}\\
%	&=\frac{g(z^*+\epsilon)-g(z^*)}{\epsilon}+\frac{E(\int_{0}^{\tau^*}h(Z_s)ds|Z_0=z^*+\epsilon)}{\epsilon}\\
%	&\rightarrow g'(z^*).
%	\end{align*}
%	as $h(y)=0.$
%\end{remark}

%\subsection{Monotone stopping problems in discrete time and the Doob decomposition}
\subsection{Monotone stopping problems in discrete time and the multiplicative Doob decomposition}\label{subsec:mono_discr_mult}
%\subsubsection{The discrete time case}
For processes $(X_n)_{n\in\N}$ with $X_n>0$ we may also consider the multiplicative Doob decomposition
\[X_n=M_nA_n\]
where, with $X_0=1,\,\A_0=\{\emptyset,\Omega\}$,
\begin{align*}
M_n&=\prod_{k=1}^{n}\frac{X_k}{E(X_k|\A_{k-1})},\,n\geq 1,\mbox{ is a mean 1-martingale},\\
A_n&=\prod_{k=1}^{n}\frac{E(X_k|\A_{k-1})}{X_{k-1}},\,n\geq 1.
\end{align*}
Optimality of the myopic stopping time may also be inferred from this multiplicative decomposition in the monotone case. As in Subsection \ref{subsec:mono_discr}, we have for any $\tau$
\[A_\tau\leq A_{\tau^*}=\sup_n A_n,\;\; A_{\min\{\tau,L\}}\leq A_{\tau^*_L}.\]
The multiplicative decomposition leads, however, to different sufficient conditions for optimality. There is also a connection to a change of measure approach. Both is discussed in Subsection \ref{append:mult_optimal} below. 

We furthermore observe that we have a monotone case problem in particular if
\begin{equation*}
E\left(\frac{X_{n+1}}{X_n}\bigg|\A_n\right)\mbox{ is non-increasing in $n$},
\end{equation*}
which turns out to be a basis for the treatment of product-type problems in the following section.

\subsection{Optimality of The Myopic Rule based on the Multiplicative Decomposition}\label{append:mult_optimal}
We assume the setting of Subsection \ref{subsec:mono_discr_mult} and work under the assumption that we are in the monotone case.

\subsubsection{Optimality for finite time horizon}	
For any bounded stopping time $\tau\leq L$
\[EX_\tau =EM_{\tau}A_\tau=EM_LA_\tau\leq EM_LA_{\tau^*_L}=EX_{\tau_L^*},\]
so we arrive as in Subsection \ref{subsec:opt_finite_add} at
\[EX_{\tau^*_L}=\sup_{\tau\leq L}EX_\tau.\]

\subsubsection{Optimality for infinite time horizon}

To extend this argument to infinite time horizon, first note that, due to the positivity property of the $X_n$, \eqref{eq:V2} is valid due to the discussion in Subsection \ref{subsec:ass}. Condition \eqref{eq:V1} has to be taken care of for the specific problem at hand (and we know from Subsection \ref{subsec:ass} that\ $E(\sup_n X_n)<\infty$ is sufficient).

We now present a measure-change approach leading to another sufficient condition for optimality. We use a probability measure $Q$ such that $\frac{dQ|_{\A_n}}{dP|_{\A_n}}=M_n$ for each $n$, invoking the Kolmogorov extension theorem for the existence of $Q$. Then for any stopping time $\tau$
\[EX_\tau1_{\{\tau<\infty\}}=E_QA_\tau1_{\{\tau<\infty\}}.\]
%We assume now that\ 
\begin{proposition}
	Assume that
	\begin{equation}\tag{W1}\label{eq:W1}
	Q(\tau^*<\infty)=1.
	\end{equation}
	Then, $\tau^*$ is optimal, i.e. $EX_{\tau^*}=\sup_{\tau}EX_\tau$
\end{proposition}
\begin{proof} For any stopping time $\tau$
		\begin{align*}
	EX_\tau&\leq \liminf_{L\rightarrow\infty}EX_{\min\{\tau,L\}}=\liminf_{L\rightarrow\infty}E_QA_{\min\{\tau,L\}}\\
	&\leq E_QA_{\tau^*}=E_QA_{\tau^*}1_{\{\tau^*<\infty\}}\\
	&=EX_{\tau^*}1_{\{\tau^*<\infty\}}\leq EX_{\tau^*}.
	\end{align*}
\end{proof}
As \eqref{eq:W1} does not seem to be very handy for applications, we now give a sufficient condition. Using
\[Q(\tau^*>n)=\int_{\{\tau^*>n\}}M_ndP=\int_{\{\tau^*>n\}}\frac{X_n}{A_n}dP\leq  \int_{\{\tau^*>n\}}X_ndP, \]
we see that we obtain optimality for $\tau^*$ if
\[\int_{\{\tau^*>n\}}X_ndP\rightarrow0\mbox{ as }n\rightarrow\infty.\]
Note the similarities to the approach of Beibel and Lerche as presented, e.g., in \cite{BL} and \cite{lu}. There, in the continuous time case, the decomposition $X_t=A_tM_t$ and
$EX_\tau1_{\{\tau<\infty\}}=E_QA_\tau1_{\{\tau<\infty\}}$ is used for the case $A_t=g(Z_t)$ for some diffusion\ $Z$. Then, the stopping time 
\[\sigma^*=\inf\{t:g(Z_t)=\sup_zg(z)\}\]
has on $\{\sigma^*<\infty\}$ the property $A_{\sigma^*}=\sup_tA_t$, as the myopic stopping time. 
%Using the probability measure $Q$ given by $\frac{dQ|_{\A_n}}{dP|_{\A_n}}=M_n$, we have for $\tau$ with $P(\tau<\infty)=1$
%\[EX_\tau=E_QA_\tau.\]
%For the myopic stopping time $\tau^*$ in a monotone case problem, it follows as in the additive Doob decomposition, under $\tau<\infty\;P-$a.s.,
%\[A_{\tau^*}=\sup_{n\in\N} A_n,\] 
%hence
%\[EX_{\tau}=E_QA_\tau\leq E_Q\sup_{n\in\N} A_n=E_QA_{\tau^*}=EX_{\tau^*}.\]
%It is clear that we have a monotone case problem in particular if
%\begin{equation*}
%E\left(\frac{X_{n+1}}{X_n}\bigg|\A_n\right)\mbox{ is non-increasing in $n$}.
%\end{equation*}

\subsection{Monotone stopping problems in continuous time and the multiplicative Doob-Meyer decomposition}\label{subsec:mono_cont_mult}
Also in continuous time, a multiplicative Doob-Meyer-type decomposition of the form 
\[X_t=M_tA_t\]
can be found in the case of a positive special semimartingale $X$, see \cite{jamshidian2007duality}. For the ease of exposition, we now concentrate on the case of continuous semimartingales to have more explicit formulas. Using ibid, Theorem 4.2, $M$ is a local martingale and if 
$(\int_0^tY_sdV_s)_{t\in[0,\infty)}$ denotes the process in the additive Doob-Meyer decomposition
as in\ Subsection \ref{subsubsec:doob_cont}, the process $A$\ here is given by
\[A_t=\exp\left(\int_0^t\frac{Y_s}{X_{s}}dV_s\right).\]
The optimality may be discussed as in the discrete time case. 
We again remark that the problem can be identified to be monotone in particular if the process 
\[\left(\frac{Y_t}{X_{t}}\right)_{t\in[0,\infty)}\mbox{ is non-increasing.}\] 
%Under weak integrability assumptions discussed in \cite{jamshidian2007duality}, the
%myopic stopping time 
%\[\tau^*=\inf\{t:Y_t\leq 0\}\]
%is then indeed optimal. 

\section{Multidimensional Monotone Case Problems}\label{sec:multi}
We now come to the main point of this paper:\ Can we use the monotone case approach to find truly multidimensional stopping problems with explicit solutions? The answer is yes \replaced{as shown in the following section by several non-trivial examples.}{in so far, as we can present nontrivial examples in the following section.}

\subsection{The sum problem}
%Belonging to the basic knowledge of any student of mathematics is the fact that 
Already for sequences of real numbers,
$\sup_n(a_n+b_n)\leq \sup_na_n+\sup_nb_n,$
usually with strict inequality. For optimal stopping, 
%this means that being able to solve the stopping problems 
%\[\sup_\tau EX^1_\tau\mbox{ and }\sup_\tau EX^2_\tau \]
%does not imply that we are able to solve the stopping problem
\[\sup_\tau E(X^1_\tau+X^2_\tau)\leq \sup_\tau E(X^1_\tau)+\sup_\tau E(X^2_\tau),\]
with strict inequality as a rule; this means that being able to solve the stopping problems for $(X^1_n)_n$ and $(X^2_n)_n$ does not imply that we are able to solve the stopping problem for $(X^1_n+X^2_n)_n$.
\subsubsection{Discrete time  case}
Now let us look at $m$ sequences $(X^1_n)_{n\in\N},\dots, (X^m_n)_{n\in\N}$, adapted to a common filtration $(\A_n)_{n\in\N}$, with Doob decompositions
\[X_n^i=X_1^i+M_n^i+A_n^i,\;i=1,\dots,m,\]
where $A_n^i=\sum_{k=1}^{n-1}Y_k^i$ as in Subsection \ref{subsec:mono_discr}. Then, the Doob decomposition for the sum process is
\[\sum_{i=1}^mX_n^i=\sum_{i=1}^mX_1^i+\sum_{i=1}^mM_n^i+\sum_{k=1}^{n-1}\sum_{i=1}^mY_k^i,\;i=1,\dots,m.\]
Now, if for each $i$ the stopping problem for $(X^i_n)_{n\in\N}$ is a monotone case problem it does not necessarily follow that we have a monotone case problem for $\left(\sum_{i=1}^mX_n^i\right)_{n\in\N}$, see Example \ref{subsec:burglar} below. But in the special case that all the $(Y^i_k)_{k\in\N}$ are non-increasing in $k$ the monotone case property holds. We formulate this as a simple proposition:
\begin{proposition}\label{prop:sum_discr}
Assume that the processes $X^1,\dots,X^k$ have Doob decompositions
\[X_n^i=X_1^i+M_n^i+\sum_{k=1}^{n-1}Y^i_k,\;i=1,\dots,m,\] 
such that all the sequences $(Y^i_k)_{k\in\N}$ are non-increasing in $k$. Then:
\begin{enumerate}[(i)]
	\item The sum problem for $\left(\sum_{i=1}^mX_n^i\right)_{n\in\N}$ is a monotone case problem with myopic rule
	\[\tau^*=\inf\{k: \sum_{i=1}^mY^i_k\leq 0\}.\]
	\item If \eqref{eq:V1} and \eqref{eq:V2} hold for $\left(\sum_{i=1}^mX_n^i\right)_{n\in\N}$, then the myopic rule $\tau^*$ is optimal. 
\end{enumerate}
\end{proposition}

\begin{proof}
	For the Doob decomposition
	\[\sum_{i=1}^mX_n^i=\sum_{i=1}^mX_1^i+\sum_{i=1}^mM_n^i+\sum_{k=1}^{n-1}\sum_{i=1}^mY_k^i,\;i=1,\dots,m,\]
	the sequence $\left(\sum_{i=1}^mY_k^i\right)_{k\in\N}$ is non-increasing in\ $k$ by assumption, yielding $(i)$.\\
	$(ii)$ now follows from $(i)$ using Proposition \ref{prop:opt_myopic_discr}.
%	the discussion in Subsection \ref{append:doob_optimal}.
\end{proof}

%\begin{remark}
%	The assumption in Proposition \ref{prop:sum_discr} $(ii)$ can be relaxed. 
%	\end{remark}

\subsubsection{Continuous time case}
Now let us look at $m$ continuous time processes $(X^1_t)_{t\in[0,\infty)},\dots, (X^m_t)_{t\in[0,\infty)}$, adapted to a common filtration $(\A_n)_{n\in\N}$, with Doob-Meyer decompositions
\[X_i^i=X_0^i+M_t^i+A_t^i,\;i=1,\dots,m,\]
where $A_t^i=\int_0^tY^i_sdV_s$ for an increasing $V$ independent of $i$. (The typical case is $dV_s=ds$.) Then, the Doob decomposition for the sum process is
\[\sum_{i=1}^mX_t^i=\sum_{i=1}^mX_0^i+\sum_{i=1}^mM_t^i+\int_0^t\sum_{i=1}^mY_s^idV_s,\;i=1,\dots,m,\]
so that for non-increasing $(Y^i_t)_{t\in[0,\infty)},\;i=1,\dots,m,$ the monotone case property holds. We obtain as in the discrete time case:

\begin{proposition}\label{prop:cont_sum}
	Assume that the processes $X^1,\dots,X^k$ have Doob decompositions
\[X_i^i=X_0^i+M_t^i+\int_0^tY_s^idV_s,\;i=1,\dots,m,\]
	such that all the processes $(Y^i_t)_{t\in[0,\infty)}$ are non-increasing in $t$. 
	\begin{enumerate}[(i)]
		\item The sum problem for $\left(\sum_{i=1}^mX_t^i\right)_{t\in[0,\infty)}$ is a monotone case problem with myopic rule
		\[\tau^*=\inf\{t: \sum_{i=1}^mY^i_t\leq 0\}.\]
		\item If \eqref{eq:V1} and \eqref{eq:V2} hold for $\left(\sum_{i=1}^mX_t^i\right)_{t\in[0,\infty)}$, then the myopic rule $\tau^*$ is optimal. 
	\end{enumerate}
\end{proposition}

\begin{remark}\label{rem:prod}
	The assumptions of the previous Propositions can obviously be relaxed by assuming that the processes $Y^i$ are of the form
	\[Y^i=B\tilde Y^i,\]
	where $\tilde Y^i$ is a non-increasing process and $B>0$ is a process independent of $i$.
\end{remark}

\subsection{The product problem}\label{prod_problem}
\subsubsection{Discrete time case}
Now let us again consider $m$ positive sequences $(X^1_n)_{n\in\N},\dots, (X^m_n)_{n\in\N}$, which we now assume to be independent. We are interested in the product problem with gain $X_n=\prod_{i=1}^nX^i_n,\,n\in\N.$ In this case\deleted{, it holds that}
\[E\left(\frac{X_{n+1}}{X_n}\bigg|\A_n\right)=\prod_{i=1}^mE\left(\frac{X_{n+1}^i}{X_n^i}\bigg|\A_n\right).\]
So, if in the special individual monotone case problems\deleted{,  it holds that}
\begin{equation}\label{eq:monot_mult}
E\left(\frac{X_{n+1}^i}{X_n^i}\bigg|\A_n\right)\mbox{ is non-increasing in $n$},
\end{equation}
then this also holds for the product. By noting that \eqref{eq:V2} is fulfilled automatically by the positivity, we obtain

\begin{proposition}\label{prop:discr_prod}
	Assume that the processes $X^1,\dots,X^k$ are positive, independent, and have multiplicative Doob decomposition
	\[X_n^i=M_n^i\prod_{j=1}^{n}\frac{E(X_j^i|\A_{j-1})}{X_{j-1}^i},\;i=1,\dots,m,\]
	such that \eqref{eq:monot_mult} holds true. Then:
	\begin{enumerate}[(i)]
		\item The product problem for $\left(\prod_{i=1}^mX^i_n\right)_{n\in\N}$ is a monotone case problem with myopic rule
		\[\tau^*=\inf\left\{n: \prod_{i=1}^mE\left(\frac{X_{n+1}^i}{X_n^i}\bigg|\A_n\right)\leq 1\right\}.\]
		\item If \eqref{eq:V1} holds for $\left(\prod_{i=1}^mX^i_n\right)_{n\in\N}$, then the myopic rule $\tau^*$ is optimal. 
	\end{enumerate}
\end{proposition}

\subsubsection{Continuous time case}
The same argument as in the discrete case yields
\begin{proposition}\label{prop:cont_prod}
	Assume that the processes $X^1,\dots,X^k$ are positive, independent  semimartingales, and have multiplicative Doob-Meyer decomposition
	\[X_t^i=M_t^i\exp\left(\int_0^t\frac{Y_s^i}{X_{s}^i}dV_s\right),\;i=1,\dots,m,\]
	such that $\left(\frac{Y_t^i}{X_{t}^i}\right)_{t\in[0,\infty)}$ is non-increasing in $t$ for $i=1,\dots,m$. Then:
	\begin{enumerate}[(i)]
		\item The product problem for $\left(\prod_{i=1}^mX^i_t\right)_{t\in[0,\infty)}$ is a monotone case problem with myopic rule
		\[\tau^*=\inf\left\{t: \prod_{i=1}^m\exp\left(\int_0^t\frac{Y_s^i}{X_{s}^i}dV_s\right)\leq 1\right\}=\inf\left\{t: \sum_{i=1}^m\int_0^t\frac{Y_s^i}{X_{s}^i}dV_s\leq 0\right\}.\]
		\item If \eqref{eq:V1} holds for $\left(\prod_{i=1}^mX^i_t\right)_{t\in[0,\infty)}$, then the myopic rule $\tau^*$ is optimal. 
	\end{enumerate}
\end{proposition}

\section{Examples}\label{sec:ex}
\subsection{The multidimensional house-selling problem}
\subsubsection{Sum problem}
\begin{enumerate}[(i)] 
	\item \emph{With recall:} 
Consider $m$ independent i.i.d. integrable sequences $(Z^1_n)_{n\in\N},\dots,(Z^m_n)_{n\in\N}$ and let for $i=1,\dots,m$ 
\[X^i_n=\max\{Z^i_1,\dots,Z^i_n\}-cn,\;c>0.\]
In the terminology of the house-selling problem we have houses $i=1,\dots,m$ to sell with gain $X_n^i$ when selling it at time $n$ describing a seller's market, the max stating that former offers may be used. It is well-known, see \cite{CRS}, that this is a monotone case problem with
\[Y^i_k=E(\max\{Z^i_1,\dots,Z^i_{k+1}\}|\A_k)-\max\{Z^i_1,\dots,Z^i_k\}-c=f_i(S^i_k)-c\]
where
\[S^i_k=\max\{Z^i_1,\dots,Z^i_k\},\;f_i(z)=E((Z^i_1-z)^+).\]
The filtration $(\A_n)_{n\in\N}$ is, of course, the one generated by the independent sequences. Since the $f_i$ are non-increasing in\ $z$ and the $S^i_k$ are non-decreasing in\ $k$, the processes $(Y^i_k)_{k\in\N}$ are non-increasing in $k$. Proposition \ref{prop:sum_discr} yields that the multidimensional sum problem with reward
\[X_n=\sum_{i=1}^{m}\max\{Z^i_1,\dots,Z^i_n\}-cn \]
is a monotone case problem, and the myopic stopping time
\[\tau^*=\inf\{k: \sum_{i=1}^mf_i(S^i_k)\leq c\}=\inf\{k: (S^1_k,...,S^m_k)\in S_m^*\}\]
with
\[S_m^*=\{(z_1,...,z_m):\sum_{i=1}^mf_i(z_i)\leq c\}\]
is optimal under the additional condition of finite variance for $Z_n^i$. The validity of \eqref{eq:V1} and \eqref{eq:V2} follows as in the univariate case, see \cite{Ferguson}, Appendix to Chapter 4, Theorem 1.

This problem was already solved by \cite{bruss1997multiple}, using the monotone case property for the sum problem in this example. Their treatment includes the validity of \eqref{eq:V1} and \eqref{eq:V2}, and also explicit solutions for the uniform distribution where $f_i(z)=\frac{1}{2}(1-z)^2,\,z\in[0,1],$ and the exponential distribution with $f_i(z)=e^{-z}$, $z\in(0,\infty)$. Then
%	\[\tau^*=\inf\{k: (S^1_k,\dots,S^m_k)\in S_m\},\]
%	with
		\[S_m^*=\{(z_1,\dots,z_m):\sum_{i=1}^m(1-z_i)^2\leq 2c\},\mbox{ resp. }=\{(z_1,\dots,z_m):\sum_{i=1}^me^{-z_i}\leq c\}.\]
	In the house-selling problem, the functions $f_i$ are decreasing and convex. In the two examples above, we assumed identical distributions, so that\ $f_i=f$ are independent of $i$ and the stopping sets are symmetrical. Of course, this will disappear for non-identical distributions.
%\end{remark}
If the underlying distributions are discrete then the $f_i$ will be piecewise linear and the optimal stopping sets polyhedrons.
\item \emph{Without recall:} In the house-selling problem without recall we have 
\[X_n^i=Z_n^i-cn,\]
which is not a monotone case problem. But this one-dimensional problem has the same solution, i.e. optimal stopping time and optimal value as the problem with recall. This is well-known and follows from 
\[Z_n^i\leq \max\{Z_1^i,...,Z_n^i\}\mbox{ and }Z_{\tau^*}^i=\max\{Z_1^i,...,Z_{\tau^*}^i\}.\]
Looking at the sum problem without recall we have
\[X_n=\sum_{i=1}^mZ_n^i-cn\]
which again is not a monotone case problem. But clearly, this is not a truly multidimensional problem as it reduces to the one-dimensional problem for $\tilde{Z}_n=\sum_{i=1}^mZ_n^i.$ Using\ $\tilde{f}(z)=E((\tilde Z_1-z)^+)$, the optimal stopping time is given by
\[\tilde{\tau}^*=\inf\{k:\tilde f(\tilde Z_n)\leq c\}=\inf\{k:(Z^1_k,...,Z^m_k)\in\tilde S_m\} \]
with
\[\tilde S_m=\{(z_1,..,z_m):z_1+...+z_m\geq \tilde f^{-1}(c)\}, \]
thus a different geometric structure than in the problem with recall. For the explicit computation of $\tilde f$, an explicit expression for the distribution of $Z_n$ is necessary. As a simple example we take two uniform distributions to arrive at 
\[\tilde f(z)=\frac{4}{3}+z^2-2z-\frac{z^3}{6},\,1\leq z\leq 2.\]
\item Since
\[\tilde f(z_1+...+z_m)=E\left(\left(\sum_{i=1}^m(Z_1^i-z_i)\right)^+\right)\leq \sum_{i=1}^m E((Z_1^i-z_i)^+)=\sum_{i=1}^m f_i(z_i) \]
it follows that
\[S_m^*\subseteq \tilde S_m, \]
see Figure \ref{fig:house_sum}. However, 
no ordering between $\tilde{\tau}^*$ and $\tau^*$ may be inferred, as they are entrance times for two different stochastic sequences.
\begin{figure}[ht]
	\centering
	\includegraphics[width=5cm,height=5cm]{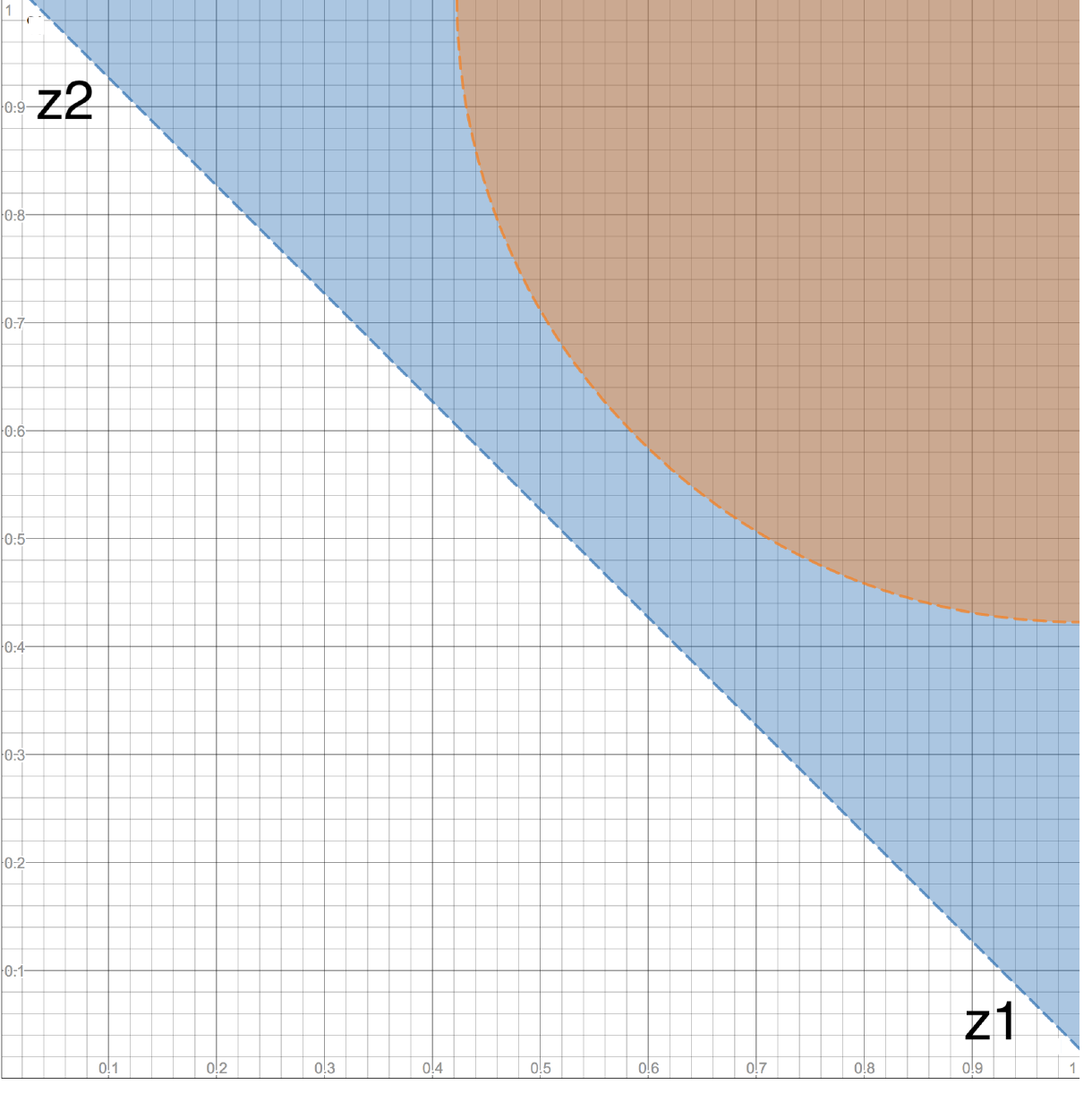}
	\caption{Stopping sets $S_m^*$ (red, with recall) and $\tilde S_m$ (blue\added{ and red}, without recall) \added{for $(Z^1,Z^2)$} in the multidimensional sum house-selling problem for uniform distributions and $c=1/6$.}\label{fig:house_sum}
\end{figure}
\end{enumerate}

\subsubsection{Product problem}
Another multidimensional version of the house-selling problem is the product problem with constant costs, that is, with reward
\[\prod_{i=1}^m\max\{Z^i_1,\dots,Z^i_n\}-cn.\]
\replaced{Looking at possible applications for the product structure, prices of houses do not provide an appropriate setting. We look at the selling of other assets, keeping the traditional notion of the house-selling problem, although \emph{asset-selling problem} might be more appropriate. Stopping problems with product structure over various assets arise in American option pricing of financial engineers, and we sketch a possible case in the following lines. Consider}
{ As an illustration consider} (with recall) a selling problem for a commodity together with a currency conversion. Here $m=2$ and the first factor gives the obtainable prices for the commodity in currency $A$, the second factor gives the obtainable exchange rates from currency $A$ to currency $B$, and the product gives the prices in currency $B$.

\replaced{In the product problem with constant costs it can straightforwardly be checked that this does not lead to a monotone case problem. }{It can, however, straightforwardly be checked that this does not lead to a monotone case problem. }
We now modify the classical problem by using a discounting factor $\rho\in(0,1)$\added{ which is also appropriate for the financial engineering setting}. More precisely, for $(Z^1_n)_{n\in\N},\dots,(Z^m_n)_{n\in\N}$ as above with $Z^i_n>0$ a.s., let for $i=1,\dots,m$ 
\[X^i_n=\rho^n\max\{Z^i_1,\dots,Z^i_n\}.\]
Then,
\begin{align*}
E\left(\frac{X_{k+1}^i}{X_k^i}\bigg|\A_k\right)&=\rho g_i(S^i_k)
\end{align*}
where
\[S^i_k=\max\{Z^i_1,\dots,Z^i_k\},\;g_i(z)=E\left(\max\left\{1,\frac{Z^i_1}{z}\right\}\right).\]
Similar as for the sum problem, the $g_i$ are decreasing in\ $z$ and the $S^i_k$ are non-decreasing in\ $k$, so that the processes $E\left(\frac{X_{k+1}^i}{X_k^i}\bigg|\A_k\right)$ are non-increasing in $k$. Therefore, the multidimensional product problem with gain
\[X_n=\prod_{i=1}^{m}X^i_n \]
is a monotone case problem and the myopic stopping time reads as
\[\tau^*=\inf\left\{k: \prod_{i=1}^mg_i(S^i_k)\leq \rho^{-m}\right\}.\]
It is not difficult to see that $\tau^*$ is optimal according to Proposition \ref{prop:discr_prod}. Indeed, \eqref{eq:V2} is clear due to the non-negativity and for \eqref{eq:V1} it can be checked that $E(\sup_n X_n^i)<\infty$ for all integrable $Z^i_n$, see e.g. \cite{Ferguson}, Chapter 4, Section 4.7. 
We provide the short argument: Since
\[E\sup_n X_n\leq E\left(\prod_{i=1}^m\sup_n X_n^i\right)=\prod_{i=1}^mE\left(\sup_n X_n^i\right) \]
it is enough to consider $m=1$. Then
\[E\sup_n X_n^1\leq E\left(\sup_n \max\{\rho Z^1_1,\dots,\rho^nZ^1_n\}\right)\leq E\left(\sum_{n=1}^\infty\rho^nZ^1_n\right)=\frac{\rho}{\rho-1}EZ^1_1<\infty.\]

As for the sum problem, we obtain an explicit optimal stopping rule when considering concrete  distributions. For example, consider the case that all $Z^i_n$ are uniformly distributed on $[0,1]$, then
\[g_i(z)=g(z)=\int_0^1\max\left\{1,\frac{u}{z}\right\}du=\frac{1+z^2}{2z},\]
so
\[\tau^*=\inf\{k: (S^1_k,\dots,S^m_k)\in S_m\},\]
where 
\[S_m=\left\{(z_1,\dots,z_m):\prod_{i=1}^m\frac{1+z_i^2}{z_i}\leq \left(\frac{\rho}{2}\right)^{-m}\right\}.\]
See Figure \ref{fig:house_prod} for an illustration.

\begin{figure}[ht]
	\centering
	\includegraphics[width=5cm,height=5cm]{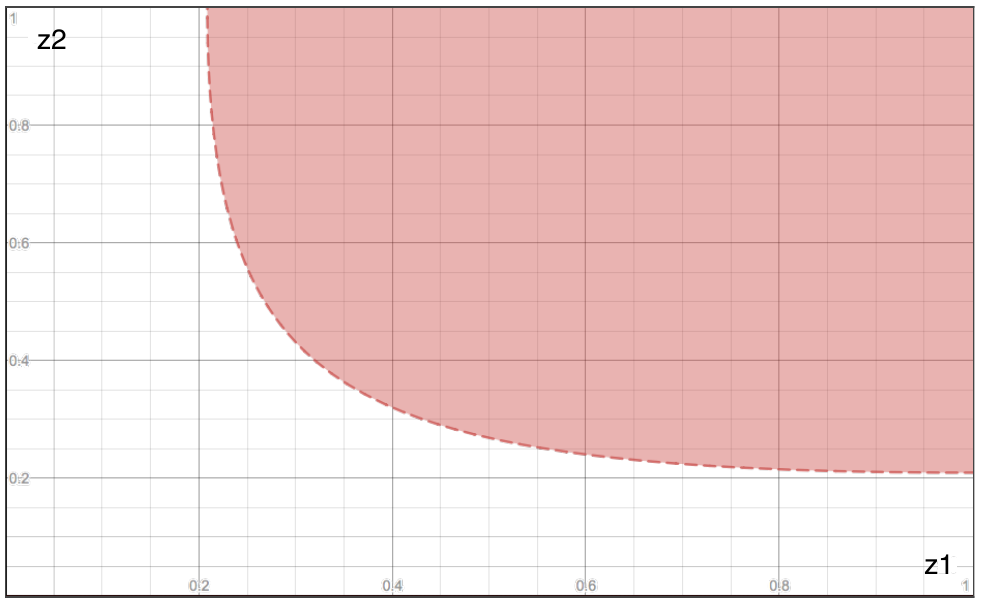}
	\caption{Stopping set \added{$S_m$ (red) for $(S^1,S^2)$} in the multidimensional product house-selling problem for uniform distributions.}\label{fig:house_prod}
\end{figure}

\begin{remark}
%	\begin{enumerate}[(i)]

%		\item 
		In the house-selling problem, as well as in other stopping problems, one might want to also study a problem of $\max$-type, that is the problem with gain
		\begin{align*}
		X_n&=\max_{i=1,\dots,m}X_n^i-cn\\
		&=\max_{i=1,\dots,m}\max\{Z_1^i,\dots,Z^i_n\}-cn\\
		&=\max_{l=1,\dots,n}\max_{i=1,\dots,m}Z_l^i-cn.
		\end{align*}
		So, this is not a truly multidimensional problem as it boils down to a one-dimensional problem for $\tilde Z_n:= \max_{i=1,\dots,m}Z_n^i$. But in general, it seems harder to work with max-type problems than with sum problems due to the nonlinearity of the max function. 
%		\item The product problem as introduced in Subsection \ref{prod_problem} is not monotone here in general. 
%	\end{enumerate}

\end{remark}

\subsection{The multidimensional burglar's problem}\label{subsec:burglar}
\subsubsection{Sum problem}
Here, we have for $i=1,\dots,m$ independent i.i.d. sequences $(Z^i_n)_{n\in\N}$ and $(\delta^i_n)_{n\in\N}$,  where $Z^i_n\geq 0$ describes the burglar's gain and $\delta^i_n=0$ or $=1$ when getting caught or not caught, resp. Then, we look at 
\[X^i_n=\left(\sum_{j=1}^nZ_j^i\right)\prod_{j=1}^{n}\delta^i_j\]
with obvious interpretation. The sum problem corresponds to the question when a burglar gang should stop their work. It is well-known that for each $i$ we have a monotone case problem. Indeed, writing $p^i=E\delta^i,\;a^i=EZ^i_1$ it holds that
\begin{align*}
Y^i_k&=E\left(\left(\sum_{j=1}^kZ^i_j+Z^i_{k+1}\right)\prod_{j=1}^k\delta^i_j\delta^i_{k+1}\bigg|\A_k\right)-X^i_k\\
&=X^i_kp^i+\left(\prod_{j=1}^k\delta^i_j\right)a^ip^i-X^i_k=X^i_k(p^i-1)+\left(\prod_{j=1}^k\delta^i_j\right)a^ip^i,
\end{align*}
hence
$Y^i_k\leq 0$ iff 
\[\prod_{j=1}^k\delta^i_j=0\mbox{ or }\sum_{j=1}^kZ_j^i\geq \frac{a^ip^i}{1-p^i}.\]

Let us first look at the sum problem for $m=2$ and constant $p^i=p,\,a^i=a$. Then,
\[Y^1_k+Y^2_k=(X^1_k+X^2_k)(p-1)+ \left(\left(\prod_{j=1}^k\delta^1_j\right)+\left(\prod_{j=1}^k\delta^2_j\right)\right)ap.\]
If $\prod_{j=1}^k\delta^1_j=1=\prod_{j=1}^k\delta^2_j$, this becomes
\[\sum_{j=1}^k(Z_j^1+Z_j^2)(p-1)+2ap, \]
hence 
\[Y^1_k+Y^2_k\leq 0 \mbox{ iff }\sum_{j=1}^k(Z_j^1+Z_j^2)\geq \frac{2ap}{1-p}.\]
But if, e.g., the next $\delta^1_{k+1}=1,\,\delta^2_{k+1}=0,$ then
\[Y^1_{k+1}+Y^2_{k+1}=\left(\sum_{j=1}^{k+1}Z_j^1\right)(p-1)+ ap\]
and $\left(\sum_{j=1}^{k+1}Z_j^1\right)\geq \frac{ap}{1-p}$ does not hold true in general. So, the sum problem is not monotone in general. 

In the case that $(\delta_n)_{n\in\N}=(\delta^i_n)_{n\in\N}$ is independent of $i$ -- that is the police takes away all stolen goods when catching one member of the gang -- the problem for 
\[\sum_{i=1}^m X_n^i=\left(\sum_{j=1}^n\sum_{i=1}^m Z^i_j \right)\prod_{j=1}^k\delta_j \]
is simply the one-dimensional case for
\[\tilde Z_j=\sum_{i=1}^m Z^i_j.\]

\subsubsection{Product problem}
We now consider the product version of the multidimensional burglar's problem. \added{Looking at possible applications we may turn to the pricing of American options in financial engineering. Here we may have asset prices together with exchange rates and default variables $\delta$, a bankruptcy replacing being caught. }We could directly apply Proposition \ref{prop:discr_prod}, but we want to cover a slightly more general case including geometric averages of the gains:
\[X_n=\prod_{i=1}^m (S_n^i\beta_n^i)^{\alpha_i},\;\alpha_i>0 \]
with
\[S_n^i=\sum_{j=1}^nZ_j^i,\;\beta_n^i=\prod_{j=1}^{n}\delta_j^i, \]
so
\[X_n=\prod_{i=1}^m {S_n^i}^{\alpha_i}\prod_{i=1}^m\beta_n^i,\;X_{n+1}=\prod_{i=1}^m (S_n^i+Z^i_{n+1})^{\alpha_i}\prod_{i=1}^m(\beta_n^i\delta_{n+1}^i). \]
Using
$\lambda=\prod_{i=1}^mp^i$ it follows
\[E(X_{n+1}|\A_n)=\lambda\prod_{i=1}^m\beta_n^i\prod_{i=1}^m\int(S_n^i+z)^{\alpha_i}P^{Z^i_1}(dz)\]
so that $E(X_{n+1}|\A_n)\leq X_n$ holds iff
\[\prod_{i=1}^m\beta_n^i=0\mbox{ or }\lambda\prod_{i=1}^m\int(S_n^i+z)^{\alpha_i}P^{Z^i_1}(dz)\leq \prod_{i=1}^m {S_n^i}^{\alpha_i}.\]
So under $\prod_{i=1}^m\beta_n^i=1$, the inequality to be considered becomes
\[\prod_{i=1}^m\int\left(1+\frac{z}{S_n^i}\right)^{\alpha_i}P^{Z^i_1}(dz)\leq \frac{1}{\lambda}. \]
Since $S_n^i$ is non-decreasing in $n$, we have a monotone case problem and the myopic stopping time is the the first entrance time for the $m$-dimensional random walk $\left(S_n^1,\dots,S_n^m\right)_{n\in\N}$ into the set
\[S_m=\left\{(y_1,\dots,y_m):\prod_{i=1}^mh_i(y_i)\leq \frac{1}{\lambda}\right\}\]
with
\[h_i(y)=\int\left(1+\frac{z}{y}\right)^{\alpha_i}P^{Z^i_1}(dz).\]
The optimality of the myopic stopping time follows as in the univariate case; see Proposition \ref{prop:discr_prod} and \cite{Ferguson}, 5.4.

\subsection{The multidimensional Poisson disorder problem}
The classical Poisson disorder problem is a change point-detection problem where the goal is to determine a stopping time $\tau$ which is as close as possible to the unobservable time $\sigma$ when the intensity of an observed Poisson process changes its value. Early treatments include \cite{MR0297028}, \cite{davis1976note}, and a complete solution was obtained in \cite{PS02}. Further calculations can be found in \cite{MR2158013}. 

Our multidimensional version of this problem is based on observing $m$ such independent processes with different change points $\sigma^1,\dots,\sigma^m.$ The aim is now to find one time $\tau$ which is as close as possible to the unobservable times $\sigma^1,\dots,\sigma^k.$
We now give a precise formulation. For each $i$, the unobservable random time $\sigma^i$ is assumed to be exponentially distributed with parameter $\lambda^i$ and the corresponding observable process $N^i$ is a counting process whose intensity switches from a constant $\mu_0^i$ to $\mu_1^i$ at $\sigma^i$. Furthermore, all random variables are independent for different $i$. We denote by $(\mathcal{F}_t)_{t\in[0,\infty)}$ the filtration given by
\[\mathcal{F}_t=\sigma(N^i_s,1_{\{\sigma^i\leq s\}}:s\leq t,\,i=1,\dots,m).\]
As $\sigma_i$ is not observable, we have to work under the subfiltration $(\mathcal{A}_t)_{t\in[0,\infty)}$ generated by $(N^1_t,\dots,N^m_t)_{t\in[0,\infty)}$ only. If we stop the process at $t$, a measure to describe the distance of $t$ and $\sigma^i$ often used in the literature is 
\[Z^i_t=1_{\{\sigma^i\geq t\}}+c_i\,(t-\sigma^i)^+\]
for some constant $c_i>0$. We also stay in this setting, although a similar line of reasoning could be applied for other gain functions also. As $Z^i$ is not adapted to the observable information $(\mathcal{A}_t)_{t\in[0,\infty)}$, we introduce the processes $X^1,\dots,X^m$ by conditioning as
\[X^i_t=E(Z^i_t|\A_t).\]
The classical Poisson disorder problem for $m=1$ is the optimal stopping problem of $Z^1$ over all $(\mathcal{A}_t)_{t\in[0,\infty)}$-stopping times $\tau$. Here, of course, we want to minimize (and not maximize) the expected distance, so that we have to make the obvious minor changes in the theory. 

We now study the corresponding problem for the sum process
% $\left(\sum_{i=1}^mX_t^i\right)_{t\in[0,\infty)}$
\[\sum_{i=1}^mX_t^i=E\left(\sum_{i=1}^m(1_{\{\sigma^i\geq t\}}+c_i\,(t-\sigma^i)^+)\big|\A_t\right),\;t\in[0,\infty).\]
Here, $\sum_{i=1}^m(1_{\{\sigma^i\geq t\}}+c_i\,(t-\sigma^i)^+)$ denotes the number of processes without a change before $t$ plus a weighted sum of the  cumulated times that have passed by since the other processes have changed their intensity. 

A possible application is a technical system consisting of $m$ components. Component $i$ changes its characteristics at a random time $\sigma^i$. After these changes, the component produces additional costs of $c_i$ per time unit. $\tau$ denotes a time for maintenance. Inspecting component $i$ before $\sigma^i$ produce (standardized) costs 1. Then, the optimal stopping problem corresponds to the following question:\ What is the best time for maintenance in this technical system?

%The optimal stopping problem for the sum process can thus be interpreted as follows: $m$ processes with different change points are considered and we want to find one stopping time which minimizes this measure for the cumulated distance of $\tau$ and the times $\sigma^1,\dots,\sigma^m$. A possible application is a technical system consisting of $m$ components and $\tau$ denotes the time for maintenance. 

The Doob-Meyer decomposition for $X^i,i\,=1,\dots,m,$ can explicitly be found in 
%\cite{herberts2004optimal}
\cite{PS02}, (2.14), 
 and is given by
\[X_t^i=X_0^i+M_t^i+\int_0^tY^i_sds,\]
where 
\[Y^i_t=-\lambda^i+(c^i+\lambda^i)\pi^i_t\]
 and $\pi^i_t$ denotes the posterior probability process
\[\pi^i_t=P(\sigma^i\leq t|\A_t).\]
The process $\pi^i$ can be calculated in terms of $N^i$ in this case, see \cite{PS02}, (2.8),(2.9). Indeed,
\[\pi^i_t=\frac{\phi^i_t}{1+\phi^i_t}\]
where
\[\phi^i_t=\lambda^i e^{(\lambda^i+\mu_0^i-\mu_1^i)t}e^{N^i_t\log(\mu_1^i/\mu^i_0)}\int_0^te^{-(\lambda^i+\mu_0^i-\mu_1^i)s}e^{-N^i_s\log(\mu_1^i/\mu^i_0)}ds. \]
In particular, it can be seen that the process $\phi^i$ is increasing in the case $\lambda^i\geq \mu_1^i-\mu_0^i\geq 0$, and therefore so is $Y^i.$ It is furthermore easily seen that the integrability assumptions in Proposition \ref{prop:cont_sum} are fulfilled. Therefore, we obtain that the optimal stopping time in the multidimensional Poisson disorder problem is -- under the assumption $\lambda^i\geq \mu_1^i-\mu_0^i\geq 0,\,i=1,\dots,m$ -- given by
\begin{align*}
\tau^*&=\inf\left\{t: \sum_{i=1}^mY^i_t\geq 0\right\}\\
&=\inf\left\{t: \sum_{i=1}^m (c^i+\lambda^i)\pi^i_t\geq \sum_{i=1}^m \lambda^i\right\},
\end{align*}
so that the optimal stopping time is a first entrance time into a half space 
\[S_m=\left\{(z_1,..,z_m):\sum_{i=1}^m (c^i+\lambda^i)z_i\geq \sum_{i=1}^m \lambda^i\right\}\]
for the $m$-dimensional posterior probability process. 

Let us underline that the elementary line of argument used here breaks down for general parameter sets, where more sophisticated techniques, such as pasting conditions, have to be applied. This, however, seems to be very hard to carry out in this multidimensional formulation, and there does not seem to be hope to obtain an explicit solution in these cases. 
It is furthermore interesting to note that Remark \ref{rem:finite_time} implies that our solution to the (multidimensional) Poisson disorder problem also solves the Poisson disorder problem in the finite time case, i.e. the optimal boundary is not time-dependent. This is, of course, in strong contrast to, e.g., the Wiener disorder problem, see \cite{gapeev2006wiener}. 

\subsection{Optimal investment problem for negative subordinators}
One of the most famous multidimensional optimal stopping problems is the optimal investment problem studied, e.g., in \cite{DS}, \cite{OS}, \cite{ho}, \cite{GS}, \cite{CI2}, \cite{NR}, and \cite{christensen-salminen}. It can be described as follows: Let\ $r>0$ a fixed discounting factor, $(L^1,\dots,L^m)$ be a $d$-dimensional L\'evy process 
%and we assume for simplicity that $e^{-r}E(e^{L_1^m})<1$ for all\ $i=1,\dots,d$
 and let furthermore $y_1,\dots,y_d\in(0,\infty)$. The optimal stopping problem can then be formulated as 
\[\sup_{\tau}E(e^{-r\tau}(1-y_1e^{L^1_\tau}-\dots-y_me^{L^m_\tau})).\]
At the investment time $\tau$ the investor gets the fixed standardized reward 1 and has to pay the sum of the costs $y_1e^{L^1_\tau},\dots,\,y_me^{L^m_\tau}$ (with the reward as num\'eraire).
In the notation of this paper we are faced with a sum problem for $\left(\sum_{i=1}^mX_t^i\right)_{t\in[0,\infty)}$ where
\[X_t^i:=e^{-rt}\left ( \frac{1}{m}-y_ie^{L^i_t}\right ).\]
In the case that $(L^1,\dots,L^d)$ is a $d$-dimensional (possibly correlated) Brownian motion with drift, it was conjectured in \cite{ho} that the optimal stopping time is a first entrance time into a half-space for the process $(e^{L^1},\dots,e^{L^m})$. But this was disproved for all nontrivial cases, see \cite{CI2}, \cite{NR}. The structure of the optimal boundary is much more complicated in this case and can be characterized as the solution to a nonlinear integral equation, also for more general L\'evy processes with only negative jumps, see \cite{christensen-salminen}. An explicit description cannot be expected to exist in general. 

In a special case, however, our theory immediately leads to an explicit solution. We assume now that $L^1,\dots,L^d$ are negative subordinators, i.e., all (standardized) cost factors have non-increasing sample paths. This case is typically implicitly excluded in the general theory. For example, the integral equation for the optimal boundary obtained in \cite{christensen-salminen} has no unique solution in this case.

 In terms of the characteristic triple, the assumption means that the jump measure $\Pi$ is concentrated on\ $(-\infty,0)^m$ and the drift vector has non-positive entries $a_1,\dots,a_m$.  Applying It\^o's formula for jump processes yields the Doob-Meyer decomposition as
\[X^i_t= X_0^i+M_t^i+\int_0^tY_s^ids,\]
where 
\[Y_s^i=e^{-rs}\left(c_ie^{L^i_s}-\frac{r}{m}\right)\]
and
\[c_i=y_i\left(r-a_i-\int_{(-\infty,0)^m}(e^{z_i}-1)\Pi(dz)\right)>0. \]
%Due to the discounting factor, the processes $Y^i$ do not have decreasing sample paths in general so that Proposition \ref{prop:cont_sum} cannot directly be applied. 
According to Remark \ref{rem:prod}, the sum problem is monotone and the myopic stopping time is 
\[\tau^*=\inf\left\{t\geq 0:\sum_{i=1}^mc_ie^{L^i_s}\leq {r}  \right\}. \]
As each $X^i$ is bounded, it is immediate that \eqref{eq:V1} and \eqref{eq:V2} are fulfilled. We obtain that in this case a stopping rule as conjectured in \cite{ho} is indeed optimal.

Even more, the minimum of $\tau^*$ and $L$ is the optimal stopping time for the investment problem with time horizon $L$ by Remark \ref{rem:finite_time}. For other underlying processes, such a time-independent solution can of course not be expected. 

%\subsection{The multidimensional proofreading problem}
%The form of the solutions to the previous problems is determined by an underlying Markovian structure, as is often the case for optimal stopping problems. For our approach this is, however, not needed. This is now illustrated by the following multidimensional version of the classical proofreading problem. 

%\begin{appendix}
%%	\section{On the Optimality of the Myopic Rule}
%
%\end{appendix}

\section*{Acknowledgments}
We thank the referees for their careful reading of the manuscript and their very helpful suggestions. 

\bibliographystyle{abbrvnat}
\bibliography{multidim_monotone_rev2}

\end{document}